\documentclass[12pt]{Paper}

 \usepackage[foot]{amsaddr}
\usepackage[framemethod=TikZ]{mdframed}
\usepackage{lipsum}

\usepackage{fancybox}
 \usepackage{enumerate}
  
\usepackage{cite}
\usepackage{scalerel}
\usepackage{amsthm,lipsum}
\usepackage{amsmath}
\usepackage{amssymb}
  \usepackage{paralist}
  \usepackage{graphics} 
  \usepackage{epsfig} 
  \usepackage{csquotes}
  \usepackage{subcaption}
\usepackage{graphicx}  
\usepackage{epstopdf}
 \usepackage[colorlinks=true]{hyperref}
\usepackage{bm}
\usepackage{siunitx}
\hypersetup{urlcolor=blue, citecolor=red}
\usepackage[margin=1.2in]{geometry}
\AfterEndEnvironment{figure}{\noindent\ignorespaces}

\newtheorem{theorem}{Theorem}[section]
\newtheorem{corollary}[theorem]{Corollary}
\newtheorem{lemma}[theorem]{Lemma}
\newtheorem{proposition}[theorem]{Proposition}

\theoremstyle{definition}
\newtheorem{remark}{Remark}[section]

\usepackage[T1]{fontenc}    
\usepackage[utf8]{inputenc} 

\renewenvironment{proof}{{\bf \noindent Proof.}}{\hfill $\square$\vspace{0.25cm}} 

\def\E{\mathcal{E}}

\def\f{\longrightarrow}

\def\N{\mathbb{N}}

\def\e{\varepsilon}
\def\l{\lambda}

\def\a{\bar{a}}
\def\<{\langle}
\def\>{\rangle}

\def\L{\mathcal{L}}

\def\e{\varepsilon}

\def\R{\mathbb{R}}

\def\O{\mathcal{O}}
\def\inte{\textnormal{int}\,}
\def\clo{\textnormal{cl}\,}
\def\epi{\textnormal{epi}\,}
\def\bdry{\textnormal{bdry}\,}

\def\proj{\textnormal{proj}\,}

\def\nvarphi{\bp\nvarphi}
\def\sp{\hspace{0.015cm}}
\def\bp{\hspace{-0.08cm}}

\def\bbbp{\hspace{-0.01cm}}

\def\df{\textnormal{dfar}}
\def\far{\textnormal{far}}
\def\proj{\textnormal{proj}}

\makeatletter
\newcommand{\Largeish}{\fontsize{14.5}{16.5}\selectfont} 
\makeatother

\makeatletter
\newcommand{\Largeishsub}{\fontsize{13.5}{15.5}\selectfont} 
\makeatother

\makeatletter
\renewcommand{\section}{\@startsection{section}{1}{0pt}%
  {1em}{1em}{\normalfont\Largeish\bfseries}}
\makeatother

\usepackage{titlesec}

\titleformat{\subsection}[block] {\normalfont\Largeishsub\bfseries}{\thesubsection}{1em}{}

\usepackage{parskip}

\makeatletter
\newcommand*\rel@kern[1]{\kern#1\dimexpr\macc@kerna}
\newcommand*\widebar[1]{%
  \begingroup
  \def\mathaccent##1##2{%
    \rel@kern{0.8}%
    \overline{\rel@kern{-0.8}\macc@nucleus\rel@kern{0.2}}%
    \rel@kern{-0.2}%
  }%
  \macc@depth\@ne
  \let\math@bgroup\@empty \let\math@egroup\macc@set@skewchar
  \mathsurround\z@ \frozen@everymath{\mathgroup\macc@group\relax}%
  \macc@set@skewchar\relax
  \let\mathaccentV\macc@nested@a
  \macc@nested@a\relax111{#1}%
  \endgroup
}
\makeatother

\begin{document}

\title[New Characterizations of Strong Convexity]{New Characterizations of Strong Convexity}

\author[C. Nour and J. Takche]{}
\author{Chadi Nour$^*$}
\address{Department of Computer Science and Mathematics, Lebanese American University, Byblos Campus, P.O. Box 36, Byblos, Lebanon}
\email{cnour@lau.edu.lb}
\thanks{$^*$Corresponding author}

\author{Jean Takche}
\email{jtakchi@lau.edu.lb}

\subjclass{52A20, 49J52, 49J53}

\keywords{Strong convexity, prox-regularity, exterior sphere condition, $S$-convexity, epi-Lipschitz property, convex analysis, nonsmooth analysis, proximal analysis}


\begin{abstract} Parallel to the main results of \cite{JCA2009} and \cite{JCA2018b}, which explore the equivalence between prox-regularity, the exterior sphere condition, and $S$-convexity, we present novel characterizations of the $r$-strong convexity property, namely, of the sets that can be expressed as the intersection of closed balls with the same radius $r > 0$.
\end{abstract}
\maketitle
\vspace{-0.7cm}
\section{Introduction} Let $A\subset\R^n$ be a nonempty and closed set. For $r>0$, we recall that the set $A$ is said to be {\it $r$-prox-regular}, if for all $a\in\bdry A$, the boundary $A$, and for {\it all} nonzero $\zeta\in N_A^P(a)$, the {\it proximal normal} to $A$ at $a$, the vector $\zeta$ is realized by an $r$-sphere, that is, \begin{equation*} \label{prox-regular} \left\<\frac{\zeta}{\|\zeta\|},x-a\right\>\leq \frac{1}{2r}\|x-a\|^2,\;\forall x\in A\;\;\;\; \left[\hbox{or equivalently}\; B\left(a+\frac{\zeta}{\|\zeta\|};r\right)\cap A=\emptyset\right],\end{equation*}
where $B(y;\rho)$ denotes the open ball of radius $\rho$ centered at $y$. For more information about prox-regularity and its related properties, including {\it positive reach}, {\it proximal smoothness}, {\it $p$-convexity}, and {\it $\varphi_0$-convexity}, refer to \cite{canino,csw,cm,fed,prt,shapiro}. On the other hand, we say that $A$ satisfies the {\it exterior $r$-sphere condition}, if for all $a\in\bdry A$, there {\it exists} a nonzero $\zeta_a\in N_A^P(a)$, such that $\zeta_a$ is realized by an $r$-sphere. It is worth noting that the exterior $r$-sphere condition, known from differential geometry, has numerous applications in various fields such as partial differential equations, optimal control, etc.; see, e.g., \cite{cannarsa}.

One can easily see that if $A$ is $r$-prox-regular then $A$ satisfies the exterior $r$-sphere condition. The converse is not necessarily true, as illustrated in \cite[Example 2.5]{JCA2009}. In this latter, a nonempty and closed set $A\subset\R^n$ satisfying the exterior $1$-sphere condition but fails to be $r$-prox-regular for {\it any} $r>0$ is provided as a counterexample. In the same paper,  Nour, Stern, and Takche proved in \cite[Corollary 3.12]{JCA2009} that if $A$ is {\it epi-Lipschitz} with compact boundary, then the two properties are equivalent. More precisely, under these assumptions, if $A$ satisfies the exterior $r$-sphere condition, then $A$ is $r'$-prox-regular for some $r'>0$. We recall that $A$ is said to be epi-Lipschitz (or {\it wedged}\,) if for every boundary point $a$, the set $A$ can be viewed in a neighborhood of $a$ and after the application of an orthogonal matrix, as the epigraph of a Lipschitz continuous function. This geometric definition, introduced by Rockafellar in \cite{rock}, can also be characterized by the nonemptiness of the topological interior of the Clarke tangent cone, which is equivalent to the pointedness of the Clarke normal cone; see \cite{clsw,rock}. Note that if $A$ is convex, then $A$ is  epi-Lipschitz if and only if $A$ has nonempty interior. Furthermore, an epi-Lipschitz set $A$ equals to the closure of its interior.

Inspired by certain convexity-type properties of the reachable sets of nonlinear control systems, Frankowska and Olech, in \cite{frankowska}, established the {\it $r$-strong convexity} (for some $r > 0$) of the integral of a multivalued mapping $M\colon [0, 1] \rightrightarrows\mathbb{R}^n$ under specific conditions. The definition used for $r$-strong convexity of $A\subset\R^n$ in \cite{frankowska} is that $A$ is the intersection of closed balls of same radius $r>0$. In the same paper, see \cite[Proposition 3.1]{frankowska} (see also \cite[Theorem 2.1]{Goncharov} and \cite[Theorem 3.3]{Nacry2025}), it is proved that a nonempty and closed set $A\subset\R^n$ is $r$-strongly convex if and only if for all $a\in\bdry A$, and for {\it all} nonzero $\zeta\in N_A^P(a)$, we have \begin{equation*} \label{strong-convex} \left\<\frac{\zeta}{\|\zeta\|},x-a\right\>\leq -\frac{1}{2r}\|x-a\|^2,\;\forall x\in A\;\;\;\; \left[\hbox{or equivalently}\; A\subset \widebar{B}\left(a-r\frac{\zeta}{\|\zeta\|};r\right)\right],\end{equation*}
where $\widebar{B}(y;\rho)$ denotes the closed ball of radius $\rho$ centered at $y$. For more information about strong convexity and its applications, refer to the introduction of the paper \cite{Nacry2025}, which provides a comprehensive historical overview of this property. We proceed to introduce a new geometric property. For $r>0$, we say that $A$ is {\it $r$-spherically supported} if for all $a\in\bdry A$, there {\it exists} a nonzero $\zeta_a\in N_A^P(a)$, such that \begin{equation*} \label{sphericallysupported} \left\<\frac{\zeta_a}{\|\zeta_a\|},x-a\right\>\leq -\frac{1}{2r}\|x-a\|^2,\;\forall x\in A\;\;\;\; \left[\hbox{or equivalently}\; A\subset \widebar{B}\left(a-r\frac{\zeta_a}{\|\zeta_a\|};r\right)\right].\end{equation*}
Now, building on the study conducted in \cite{JCA2009}, a natural question arises: Is there an equivalence between strong convexity and the spherical support property defined above? If not, under what {\it minimal} conditions can this equivalence be established? One can easily see that the equivalence between strong convexity and the spherical support property does not hold in general. Indeed, for $A$ the unit circle in the plane $\R^2$, we clearly have that $A$ is $1$-spherically supported but it fails to be $r$-strongly convex for any $r>0$. The first main result of this paper addresses this question through the following theorem. We prove that if the {\it interior} of $A$, denoted by $\inte A$, is {\it nonempty}, then the equivalence between strong convexity and the spherical support property holds.
\begin{theorem} \label{th1} Let $A\subset\R^n$ be a nonempty and closed set not reduced to a singleton, and let $r>0$. Then $A$ is $r$-strongly convex if and only if $A$ is $r$-spherically supported with $\inte A\not=\emptyset$.
\end{theorem}

As a corollary, we derive the following result, demonstrating that, similar to the equivalence between prox-regularity and the exterior sphere condition established in \cite{JCA2009},  the epi-Lipschitzness of $A$ is sufficient for the equivalence between strong convexity and the spherical support property. Note that the equivalence between prox-regularity and the exterior sphere condition does {\it not} hold if $\inte A\not=\emptyset$ as can be easily seen in  \cite[Example 2.5]{JCA2009}.

\begin{corollary} \label{coro0} Let $A\subset\R^n$ be a nonempty and closed set not reduced to a singleton, and let $r>0$. Then $A$ is $r$-strongly convex if and only if $A$ is epi-Lipschitz  and $r$-spherically supported.
\end{corollary}

Another important corollary is the following, where we prove that the epi-Lipschitzness condition of $A$ in Corollary \ref{coro0} can be replaced by its convexity. Notably, this corollary is discussed in \cite{weber} (refer to condition $(\hbox{ii}')$ following the proof of \cite[Theorem 2.1]{weber}). Our proofs of Theorem \ref{th1}, Corollary \ref{coro0}, and Corollary \ref{coro1} build upon nonsmooth analysis techniques, some of which were introduced in \cite{JCA2009}, combined with carefully chosen tools from convex analysis. These employed nonsmooth analysis tools play a crucial role in the proof of Theorem \ref{th1}, given that the set  $A$ is {\it only} assumed  to have a nonempty interior. Consequently, our approach differs from the techniques employed in \cite{weber}. Moreover, we are unable to discern how Theorem \ref{th1} can be derived from the results presented in \cite{weber}. 

\begin{corollary} \label{coro1} Let $A\subset\R^n$ be a nonempty and closed set, and let $r>0$. Then $A$ is $r$-strongly convex if and only if $A$ is convex and $r$-spherically supported. 
\end{corollary}

In \cite{JCA2018b}, Nour and Takche introduced, for nonempty and closed subsets of $\R^n$, a new regularity class called {\it $S$-convexity}. A nonempty and closed set $A\subset \R^n$ is said to be $S$-convex, for $S\supset A$, if no point in $S$ is the endpoint of  two normal segments to $A$ at two different boundary points of $A$, with both segments contained in $S$. By a normal segment to $A$ at $a\in\bdry A$, we mean the segment $[a,a+t\zeta]$ where $\zeta\in N_A^P(a)$ is unit and $t\geq0$. The authors proved in \cite{JCA2018b} that this new regularity class encompasses, for a suitable choice of $S$, several well-known regularity properties, including prox-regularity, the exterior sphere condition, and the union of closed balls property. For instance, they proved in \cite[Theorem 3.7]{JCA2018b} that $A$ is $r$-prox-regular if and only if $A$ is $\O_r(A)$-convex, where $\O_r(A)$ is the tube $\O_r(A):=\{x\in\R^n : d_A(x)< r\}$, and $d_A(\cdot)$ denotes the distance function to $A$. In order to extend this result to strong convexity, we introduce the following property. For $S\subset\R^n$, a nonempty and closed set $A\subset \R^n$ is said to be {\it $r$-negatively $S$-convex}, if no point in $S$ is the endpoint of two normal segments to $A$ at two different boundary points of $A$, where the lengths of both segments are greater than  $r$. By a negative normal segment to $A$ at $a\in\bdry A$, we mean the segment $[a,a-t\zeta]$ where $\zeta\in N_A^P(a)$ is unit and $t\geq 0$. So, we have the following theorem, which constitutes the second main result of this paper. For $x\in\R^n$ and $A\subset\R^n$ a nonempty and closed set, $\df_A(x):=\sup_{a\in A}\|x-a\|$ denotes the {\it farthest distance} from $x$ to $A$.

\begin{theorem}\label{th2} Let $A\subset\R^n$ be a nonempty and closed set, and let $r>0$. Then $A$ is $r$-strongly convex if and only if $A$ is convex, bounded, and $r$-negatively $\E_r(A)$-convex, where \begin{equation}\label{Er(A)} \E_r(A):=\{x\in\R^n : \df_A(x)>r\}.\end{equation}
\end{theorem}

The layout of the paper is as follows. In the next section, we introduce our notations and fundamental definitions, along with some key results from nonsmooth and convex analysis. Section \ref{mainproofs} is dedicated to the proofs of our main results.

\section{Preliminaries} We begin by introducing the fundamental notations and definitions that will be employed throughout the paper. 
\begin{itemize}
\item We denote by $\|\cdot\|$, $\<,\>$, $B$ and $\widebar{B}$, the Euclidean norm, the usual inner product, the open unit ball and the closed unit ball, respectively. For $r>0$ and $x\in\R^n$, we set $B(x;\rho):= x + \rho B$ and $\widebar{B}(x;\rho):= x + \rho \widebar{B}$.
\item For a set $A\subset\R^n$, $A^c$, $\inte A$, $\bdry A$ and $\clo A$ are the complement (with respect to $\R^n$), the interior, the boundary and the closure of $A$, respectively. 
\item The closed segment (resp. open segment)  joining two points $x$ and $y$ in $\R^n$ is denoted by $[x,y]$ (resp. $]x,y[$).
\item For $\Omega\subset\R^n$ open and $f\colon\Omega\f\R\cup\{-\infty,+\infty\}$ an extended real-valued function, we denote by $\epi f$ the epigraph of $f$.
\item The distance from a point $x$ to a nonempty and closed set $A\subset\R^n$ is denoted by $d_A(x):=\inf_{a\in A}\|x-a\|$. We also denote by $\proj_A(x)$ the set of closest points in $A$ to $x$, that is, the set of points $a$ in $A$ satisfying $d_A(x)=\|a-x\|$.
\item The farthest distance from a point $x$ to a nonempty and closed set $A\subset\R^n$ is denoted by $\df_A(x):=\sup_{a\in A}\|x-a\|$. We also denote by $\far_A(x)$ the set of farthest points in $A$ to $x$, that is, the set of points $a$ in $A$ satisfying $\df_A(x)=\|a-x\|$. 
\end{itemize}

We proceed to present some definitions and results from nonsmooth analysis, with the monographs \cite{clsw,mord,penot,rockwet,thibault} serving as our general references for these concepts. For $A\subset\R^n$ a nonempty and closed set, and for $a\in A$, we define the {\it proximal normal cone} to $A$ at $a$, denoted by $N_A^P(a)$, as $$N_A^P(a):=\{\zeta\in\R^n : \exists \sigma\geq 0,\,\<\zeta,x-a\>\leq \sigma\|x-a\|^2,\;\forall x\in A\}.$$
When $A$ is convex, the proximal normal cone $N_A^P(\cdot)$ is known to coincide with the {\it standard} normal cone in the sense of convex analysis, denoted by  $N_A(\cdot)$, and defined for $a\in A$ as $$N_A(a):=\{\zeta\in\R^n : \<\zeta,x-a\>\leq 0,\;\forall x\in A\}.$$ In this convexity case, the normal cone $N_A(a)\not=\{0\}$ for all $a\in\bdry A$. 

The following assertions are true for any nonempty and closet set $A\subset\R^n$:
\begin{itemize}
\item For all $a\in A$ and $x\in\R^n$, $$a\in\proj_A(x)\implies \begin{cases}a\in\bdry A,\\[2pt]\<x-a,y-a\>\leq \frac{1}{2}\|y-a\|^2,\;\forall y\in A,\\[2pt]  x-a\in N_A^P(a),\;\hbox{and}\,\;\proj_A(a+t(x-a))=\{a\},\;\forall t\in[0,1[.\end{cases}$$
\item For all $a\in A$ and $x\in\R^n$, \begin{equation} \label{farnormal} a\in\far_A(x)\implies  \begin{cases}a\in\bdry A,\\[2pt] \<a-x,y-a\>\leq -\frac{1}{2}\|y-a\|^2,\;\forall y\in A,\\[2pt] a-x\in N_A^P(a),\;\hbox{and}\,\;\far_A(x+t(x-a))=\{a\},\;\forall t>0.\end{cases}\end{equation}
\end{itemize}

In the following lemma, which corresponds to \cite[Exercise 3.6.5]{clsw}, we present an important characterization of the epi-Lipschitz property.

\begin{lemma}[{\bbbp\cite[Exercise 3.6.5]{clsw}}]  \label{epilemma} Let $A\subset\R^n$ a nonempty and closed set. Then, $A$ is epi-Lipschitz if and only if for all $a\in\bdry A$ there exist $v\in\R^n$ and $\e>0$ such that $$a'+tw\in A,\; \forall a'\in A\cap B(a,\e),\;\forall t\in[0,\e[,\;\forall w\in B(v;\e).$$
\end{lemma} 

The following lemma, which aligns with \cite[Lemma 3.8]{JCA2009} and was pivotal in establishing the equivalence between prox-regularity and the exterior sphere condition when $A$ is epi-Lipschitz, will also be instrumental in the proof of Theorem \ref{th1}.

\begin{lemma}[{\bbbp\cite[Lemma 3.8]{JCA2009}}] \label{lem1} Let $f\colon U\f\R$ be a $K$-Lipschitz function defined on an open, convex and bounded set $U\subset \R^n$, and let $r>0$. Assume that $\epi f$ satisfies the exterior $r$-sphere condition. Then $\epi f$ is $\frac{r}{(1+K^2)^{\frac{3}{2}}}$-prox-regular.
\end{lemma}

We now present some useful characterizations of $r$-strongly convex sets, as defined earlier in the introduction. The proofs, along with other characterizations, can be found in the papers \cite{balashov,frankowska,Goncharov,Nacry2025}.

\begin{proposition}\label{prop0} Let $A\subset\R^n$ be a nonempty and closed set, and let $r>0$. The following assertions are equivalent\sp$:$
\begin{enumerate}[$(i)$]
\item $A$ is $r$-strongly convex.
\item For all $a\in\bdry A$, and for all nonzero $\zeta\in N_A^P(a)$, we have $$ \left\<\frac{\zeta}{\|\zeta\|},x-a\right\>\leq -\frac{1}{2r}\|x-a\|^2,\;\forall x\in A.$$
\item For all $a\in\bdry A$, and for all nonzero $\zeta\in N_A^P(a)$, we have $$A\subset \widebar{B}\left(a-r\frac{\zeta}{\|\zeta\|};r\right).$$
\item For all $a\in\bdry A$, and for all nonzero $\zeta\in N_A^P(a)$, we have $$a\in \far_A\left(a-r\frac{\zeta}{\|\zeta\|}\right). $$
\item $A$ is convex and $\far_A(x)$ is a singleton for all $x\in\E_r(A)$, where $\E_r(A)$ is the set defined in \eqref{Er(A)}.
\end{enumerate}
\end{proposition}

\begin{remark} For $A\subset\R^n$ a nonempty and compact set, and for $r>0$, one can easily see, using \eqref{farnormal}, that $$\E_r(A)=\bigcup_{\substack{a\in\bdry A\\\zeta\in N_A^A(a)\,\textnormal{unit}}} \{a-t\zeta : t>r\}.$$
\end{remark}

We conclude this section with an important result from convex analysis (see \cite[Theorem 2.2.6]{schneider}).  A boundary point $a$ of  a convex set $A$ is said to be {\it regular} if there exists a unit vector $\xi_a\in \R^n$ such that $$N_A(a)=\{\l\xi_a : \l\geq 0\}. $$  Moreover,  a supporting halfspace of a convex set $A$ is called {\it regular} if its boundary contains at least one regular boundary point of $A$.

\begin{proposition}[{\bbbp\cite[Theorem 2.2.6]{schneider}}]  \label{prop1} Let $A$ be a convex body, that is, a compact and convex set with nonempty interior. Then $A$ is the intersection of its regular supporting halfspaces. 
\end{proposition}

\section{Proofs of the main results} \label{mainproofs} In this section, we present the proofs of Theorem \ref{th1}, Corollary \ref{coro0}, Corollary \ref{coro1}, and Theorem \ref{th2}, starting with the proof of Theorem \ref{th1}.

\subsection*{Proof of Theorem \ref{th1}} Let $A\subset\R^n$ be a nonempty and closed set not reduced to a singleton, and let $r>0$. If $A$ is $r$-strongly convex, then $A$ is convex with $\inte A\not=\emptyset$ (since $A$ is not a singleton). This yields that $N_A^P(a)=N_A(a)\not=\{0\}$ for all $a\in \bdry A$. Hence, for all $a\in\bdry A$, there exists a nonzero vector $\zeta_a\in N_A^P(a)$, such that $$ \left\<\frac{\zeta_a}{\|\zeta_a\|},x-a\right\>\leq -\frac{1}{2r}\|x-a\|^2,\;\forall x\in A.$$
Therefore, $A$ is $r$-spherically supported.

We proceed to prove the converse. We assume that $A$ is $r$-spherically supported with $\inte A\not=\emptyset$. Then there exist $a_0\in A$ and $\rho_0>0$ such that $$a_0\in B(a_0;\rho_0)\subset \inte A.$$

\begin{lemma} \label{epiofA} $A$ is epi-Lipschitz.
\end{lemma}
\begin{proof} We will use a proof by contradiction. Assume that $A$ is not  epi-Lipschitz. By Lemma \ref{epilemma}, there exists $a\in\bdry A$ such that for all $v\in\R^n$ and for all $\e>0$, one can find $a'\in A\cap B(a;\e)$, $t\in[0,\e[$ and $w\in B(v;\e)$ satisfying $$a'+tw\not\in A.$$
This yields, for $v:=a_0-a$ and $\e:=\frac{1}{n}$, the existence of $a_n\in A\cap  B(a;\frac{1}{n})$, $t_n\in[0,\frac{1}{n}[$ and $w_n\in B(v;\frac{1}{n})$ such that \begin{equation}\label{gene1}  a_n+t_nw_n\not\in A,\;\forall n\in\N^*.\end{equation}
We have, for $n$ sufficiently large, that  $$\|a_n+w_n-a_0\|= \|a_n+w_n-v-a\|\leq \|a_n-a\|+\|w_n-v\|\leq \frac{2}{n}<\rho_0.$$
Hence, $a_n+w_n\in B(a_0;\rho_0)\subset\inte A$ for $n$ sufficiently large. Combining this latter with \eqref{gene1}, we deduce that, for $n$ sufficiently large, there exits $\tau_n\in ]t_n,1[$ such that $$b_n:=a_n+\tau_nw_n\in\bdry A.$$
Since $A$ is  $r$-spherically supported, we get the existence, for $n$ sufficiently large, of a nonzero vector $\zeta_n\in N_A^P(b_n)$, see Figure \ref{Fig0},  such that $$\left\<\frac{\zeta_n}{\|\zeta_n\|}, x-b_n\right\>\leq -\frac{1}{2r}\|x-b_n\|^2,\;\forall x\in A.$$
\begin{figure}[tb!]
\centering
\includegraphics[scale=1.1]{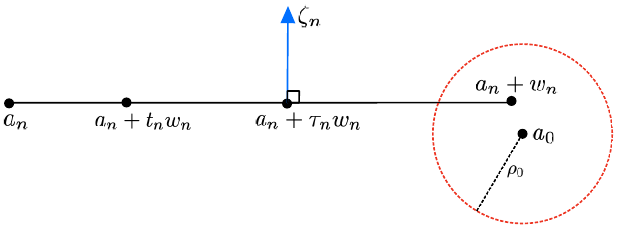}
\caption{\label{Fig0} Proof of Lemma \ref{epiofA}}
\end{figure}
This yields that, for $n$ sufficiently large $$\left\<\frac{\zeta_n}{\|\zeta_n\|}, a_n-b_n\right\>\leq -\frac{1}{2r}\|a_n-b_n\|^2\;\;\hbox{and}\;\;\left\<\frac{\zeta_n}{\|\zeta_n\|}, a_n+w_n-b_n\right\>\leq -\frac{1}{2r}\|a_n+w_n-b_n\|^2.$$
Then, for $n$ sufficiently large $$\left\<\frac{\zeta_n}{\|\zeta_n\|}, -w_n\right\>\leq -\frac{\tau_n}{2r}\|w_n\|^2\;\;\hbox{and}\;\;\left\<\frac{\zeta_n}{\|\zeta_n\|}, w_n\right\>\leq -\frac{(1-\tau_n)}{2r}\|w_n\|^2.$$
 Hence, for $n$ sufficiently large, we have $$0\leq\frac{\tau_n}{2r}\|w_n\|^2\leq \left\<\frac{\zeta_n}{\|\zeta_n\|}, w_n\right\>\leq -\frac{(1-\tau_n)}{2r}\|w_n\|^2\leq0.$$
 Since $v\not=0$, we can assume that $w_n\not=0$ for $n$ sufficiently large. Then,  for $n$ sufficiently large, we have $$\tau_n\leq -(1-\tau_n),\;\hbox{and hence}\;\,0\leq -1,$$
 which gives the desired contradiction. Therefore, $A$ is epi-Lipschitz. \end{proof}

Now since $A$ is $r$-spherically supported, we have, for every $a\in\bdry A$, the existence of a nonzero vector $\zeta_a\in N_A^P(a)$, such that $$ \left\<\frac{\zeta_a}{\|\zeta_a\|},x-a\right\>\leq -\frac{1}{2r}\|x-a\|^2,\;\forall x\in A \;\;\;\;\left[\hbox{or equivalently}\; A\subset \widebar{B}\left(a-r\frac{\zeta_a}{\|\zeta_a\|};r\right)\right].$$
Then, $A$ is compact, and for every $a\in\bdry A$, there exists a nonzero vector $\zeta_a\in N_A^P(a)$, such that $$ \left\<\frac{\zeta_a}{\|\zeta_a\|},x-a\right\>\leq -\frac{1}{2r}\|x-a\|^2\leq \frac{1}{2\rho}\|x-a\|^2,\,\;\forall x\in A\;\hbox{and}\;\forall \rho>0.$$
This yields that $A$ satisfies the exterior $\rho$-sphere condition for any $\rho>0$.  

\begin{lemma} \label{Alocallyconvex} $A$ is locally convex.
\end{lemma}

\begin{proof} Let $a\in\bdry A$. Since $A$ is epi-Lipschitz, there exist $\delta_a>0$ and $K_a>0$ such that $\widebar{B}(a;\delta_a)\cap A$ can be viewed as the epigraph of $K_a$-Lipschitz function. Then, as $A$  satisfies the $\rho$-exterior sphere condition for any $\rho>0$, we deduce, using Lemma \ref{lem1}, that $\widebar{B}(a;\delta_a)\cap A$ is $\frac{\rho}{(1+K_a^2)^{\frac{3}{2}}}$-prox-regular for any $\rho>0$ in the following sense: For all $x\in (\bdry A)\cap \widebar{B}(a;\delta_a)$ and for all nonzero vectors $\zeta\in N^P_A(x)$, we have $$\left\<\frac{\zeta}{\|\zeta\|},y-x\right\>\leq\frac{(1+K_a^2)^{\frac{3}{2}}}{2\rho} \|y-x\|^2,\;\forall y\in \widebar{B}(a;\delta_a)\cap A,\;\forall \rho>0.$$
Taking $\rho\f+\infty$, we deduce that for all $x\in (\bdry A)\cap \widebar{B}(a;\delta_a)$ and for all $\zeta\in N^P_A(x)$, we have \begin{equation}\label{souman}\<\zeta,y-x\>\leq0,\;\forall y\in \widebar{B}(a;\delta_a)\cap A.\end{equation}
We claim that $$N_A^P(x)\cap -N_{ \widebar{B}(a;\delta_a)}^P(x)=\{0\},\;\forall x\in \widebar{B}(a;\delta_a)\cap A. $$
Indeed, if not then  there exist $x\in \bdry A$ and a unit vector $\zeta\in  N_A^P(x)$ such that $$\|x-a\|=\delta_a\,\;\hbox{and}\;\,\zeta=\frac{a-x}{\delta_a}.$$
Hence, using \eqref{souman}, we get that $$\frac{1}{\delta_a}\<a-x,a-x\>=\frac{1}{\delta_a}\|a-x\|^2=\delta_a\leq 0,$$ 
which gives the desired contradiction. Now applying \cite[Theorem 4.10]{fed}, we deduce the existence of $\eta_a>0$ such that $$\widebar{B}(a;\delta_a)\cap A\;\hbox{is}\;\frac{\rho \eta_a}{2(1+K_a^2)^{\frac{3}{2}}}\hbox{-prox-regular for all}\;\rho>0.$$ This gives that $\widebar{B}(a;\delta_a)\cap A$ is $\varrho$-prox-regular for all $\varrho>0$, and hence, $\widebar{B}(a;\delta_a)\cap A$ is convex. Therefore, $A$ is locally convex. \end{proof}

Now, we consider $(A_i)_{i\in I}$ the family of the connected components of $A$. It is known that for each $i\in I$, we have $A_i$ is nonempty, closed (and then compact), and connected. Moreover, $(A_i)_{i\in I}$ is a partition of $A$, that is, $$A_i\cap A_j=\emptyset,\;\forall i\not=j\in I,\;\hbox{and}\;\bigcup_{i\in I} A_i=A.$$

\begin{lemma}\label{lem2} We have$\sp:$ \begin{enumerate}[$(i)$]
\item $\inte A=\bigcup_{i\in I} \inte A_i$ and $\bdry A=\bigcup_{i\in I}\bdry A_i$. 
\item $A_i$ is convex for all $i\in I$.
\item $A_i$ is $r$-strongly convex for all $i\in J:=\{j\in I : \inte A_j\not=\emptyset\}$. 
\end{enumerate}
\end{lemma}
\begin{proof} $(i)$: Clearly we have $$\bigcup_{i\in I} \inte A_i \subset \inte A.$$ For the reverse inclusion, let $a\in\inte A$. Then there exists $\rho>0$ such that $B(a;\rho)\subset A$. Since the open ball $B(a;\rho)$ is connected, we deduce the existence of $j\in I$ such that $a\in B(a;\rho)\subset A_j$. This yields that $$a\in\inte A_j\subset \bigcup_{i\in I} \inte A_i.$$ 

We proceed to prove that $\bdry A=\bigcup_{i\in I}\bdry A_i$. We have \begin{eqnarray*} \bdry A=A\cap (\inte A)^c&=& \left(\bigcup_{i\in I} A_i\right)\cap \left( \bigcap_{j\in I} (\inte A_j)^c \right)\\&=&\bigcup_{i\in I} \left( A_i \cap \bigcap_{j\in I} (\inte A_j)^c \right)\\&=& \bigcup_{i\in I} A_i \cap (\inte A_i)^c\qquad[\hbox{since}\;A_i\subset (\inte A_j)^c\,\;\hbox{for}\;i\not=j]\\&=& \bigcup_{i\in I}\bdry A_i. \end{eqnarray*}
$(ii)$: Let $i\in I$. Since $A_i$ is connected and using Tietze-Nakajima theorem, see \cite{Tietze,Nakajima}, it is sufficient to prove that $A_i$ is locally convex. We fix $a\in\bdry A_i$. Then $a\in\bdry A$, and hence by Lemma \ref{Alocallyconvex}, there exists $\delta_a>0$ such that $\widebar{B}(a;\delta_a)\cap A$ is convex. Let $x, y\in \widebar{B}(a;\delta_a)\cap A_i\subset \widebar{B}(a;\delta_a)\cap A$. We have $$[x,y]\subset \widebar{B}(a;\delta_a)\cap A= \bigcup_{j\in I} \widebar{B}(a;\delta_a)\cap A_j.$$
As all the points of the interval $[x,y]$ belong to the same connected component of $A$, we deduce that $[x,y]\subset A_i$, and hence,  $[x,y]\subset \widebar{B}(a;\delta_a)\cap A_i$. Therefore, $A_i$ is convex.\vspace{0.1cm}\\
$(iii)$ Let $i\in J$. For $a\in\bdry A_i$, we have $a\in\bdry A$, and hence, there exists a nonzero vector $\zeta_a\in N_A^P(a)\subset N_{A_i}^P(a)=N_{A_i}(a)$, such that $$A_i\subset A\subset \widebar{B}\left(a-r\frac{\zeta_a}{\|\zeta_a\|};r\right).$$
This yields that \begin{equation}\label{eq1} A_i\subset\bigcap_{a\in\bdry A_i} \widebar{B}\left(a-r\frac{\zeta_a}{\|\zeta_a\|};r\right). \end{equation}
Having $\inte A_i\not=\emptyset$, we deduce that $A_i$ is a convex body. Hence, by Proposition \ref{prop1}, we have $$A_i=\bigcap_{\substack{a\in\bdry A_i\\a\,\textnormal{is regular}\\\xi_a\in N_{A_i}(a)\,\textnormal{unit}}} \{x\in\R^n : \<\xi_a,x-a\>\leq 0\}.$$
Since for $a\in\bdry A_i$ regular, we have $\xi_a=\frac{\zeta_a}{\|\zeta_a\|}$, we deduce that \begin{eqnarray*} A_i&=&\bigcap_{\substack{a\in\bdry A_i\\a\,\textnormal{is regular}}} \{x\in\R^n : \<\zeta_a,x-a\>\leq 0\}\\ &\supset& \bigcap_{\substack{a\in\bdry A_i\\a\,\textnormal{is regular}}}  \widebar{B}\left(a-r\frac{\zeta_a}{\|\zeta_a\|};r\right)\;\supset  \bigcap_{a\in\bdry A_i} \widebar{B}\left(a-r\frac{\zeta_a}{\|\zeta_a\|};r\right). \end{eqnarray*}
Combining this latter with \eqref{eq1}, we conclude that  \begin{equation*}A_i=\bigcap_{a\in\bdry A_i} \widebar{B}\left(a-r\frac{\zeta_a}{\|\zeta_a\|};r\right). \end{equation*}
Therefore, $A_i$ is $r$-strongly convex. \end{proof}

\begin{lemma}\label{lem3} The set of indices $J$ is a singleton. 
\end{lemma}
\begin{proof} Since $A$ is epi-Lipschitz, we have that $A=\clo(\inte A)$. This yields, using the nonemptiness of $A$ and Lemma \ref{lem2}$(i)$, that $J$ is nonempty. If $J$ is not a singleton, then $A$ has at least two connected components $A_i$ and $A_j$ that are convex bodies. Let $a_j\in A_j$. There exists $a_i\in\bdry A_i$ regular such that \begin{equation} \label{eq3} a_j\not\in\{x\in\R^n : \<\xi_{a_i},x-a_i\>\leq 0\},\end{equation} where $\xi_{a_i}$ is unit satisfying $N_{A_i}(a_i)=\{\l\xi_{a_i} : \l\geq 0\}$, see Figure \ref{Fig1}.
\begin{figure}[tb!]
\centering
\includegraphics[scale=0.7]{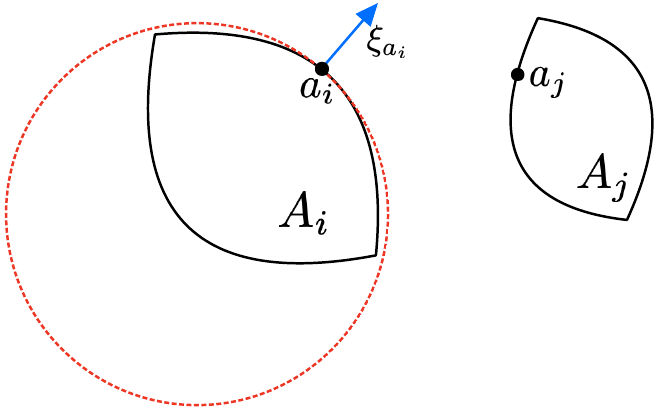}
\caption{\label{Fig1} Proof of Lemma \ref{lem3}}
\end{figure}
To see this, it is sufficient to assume the negation and use Proposition \ref{prop1} to deduce that $$a_j\in \bigcap_{\substack{a\in\bdry A_i\\a\,\textnormal{is regular}\\\xi_a\in N_{A_i}(a)\,\textnormal{unit}}} \{x\in\R^n : \<\xi_a,x-a\>\leq 0\}=A_i,$$
which contradicts $A_i\cap A_j=\emptyset$. Now, since $N_A^P(a_i)\subset N_{A_i}(a_i)$ and $N_A^P(a_i)\not=\{0\}$ (as $A$ is $r$-spherically supported), we get  that $$N_A^P(a_i)=N_{A_i}(a_i)=\{\l\xi_{a_i} : \l\geq 0\}.$$
Add to this that $A$ is $r$-spherically supported, we deduce that $$\<\xi_{a_i},x-a_i\>\leq -\frac{1}{2r}\|x-a_i\|^2,\;\,\forall x\in A.$$
Then, $$\<\xi_{a_i},a_j-a_i\>\leq -\frac{1}{2r}\|a_j-a_i\|^2<0,$$
which contradicts \eqref{eq3}. Hence, $J$ is a singleton $\{i_0\}$. \end{proof}

Now, from Lemma \ref{lem2}$(i)$, we have that $\inte A=\inte A_{i_0}$. This yields that $$A=\clo(\inte A)=\clo(\inte A_{i_0})=A_{i_0}.$$ 
Therefore, by Lemma \ref{lem2}$(iii)$, we conclude that $A$ is $r$-strongly convex. The proof of Theorem \ref{th1} is terminated. \hfill $\square$

\subsection*{Proof of Corollary \ref{coro0}} Let $A\subset\R^n$ be a nonempty and closed set not reduced to a singleton, and let $r>0$.  If $A$ is $r$-strongly convex, then using Theorem \ref{th1}, we have that $A$ is $r$-spherically supported with $\inte A\not=\emptyset$. Add to this latter that $A$ is convex, we deduce that $A$ is epi-Lipschitz. For the converse, assume that $A$ is epi-Lipschitz and $r$-spherically supported. Then, $A$ is $r$-spherically supported and $A=\clo(\inte A)$. This yields, since $A$ is nonempty, that $A$ is $r$-spherically supported with $\inte A\not=\emptyset$. By Theorem \ref{th1}, we conclude that $A$ is $r$-strongly convex.  The proof of Corollary \ref{coro0} is terminated. \hfill $\square$

\subsection*{Proof of Corollary \ref{coro1}}  Let $A\subset\R^n$ be a nonempty and closed set, and let $r>0$. Using Theorem \ref{th1}, we clearly have that if $A$ is $r$-strongly convex, then $A$ is convex and $r$-spherically supported. For the converse, assume that $A$ is convex and $r$-spherically supported. If $A$ is a singleton, then we directly deduce that $A$ is $r$-strongly convex. Now, we assume that $A$ is not a singleton. We claim that $\inte A\not=\emptyset$. Indeed, if not then $A=\bdry A$. Since $A$ is not a singleton, $A$ contains two distinct points $a$ and $a'$. The convexity of $A$ yields that $c:=\frac{a+a'}{2}\in A=\bdry A$. Then, using that $A$ is  $r$-spherically supported, we deduce the existence of a unit vector $\zeta_c\in N_A(c)$ such that \[ \left\<\zeta_c,x-c\right\>\leq -\frac{1}{2r}\|x-c\|^2,\;\forall x\in A.\]
This gives that \[\<\zeta_c,a-c\>\leq -\frac{1}{2r}\|a-c\|^2\;\;\hbox{and}\;\; \<\zeta_c,a'-c\>\leq -\frac{1}{2r}\|a'-c\|^2.\]
Hence, \[0<\frac{1}{4r}\|a-a'\|^2\leq \<\zeta_c,a-a'\>\leq -\frac{1}{4r}\|a-a'\|^2<0,\]
which gives the desired contradiction. Therefore, $\inte A\not=\emptyset$. This yields, using Theorem \ref{th1}, that $A$ is $r$-strongly convex. The proof of Corollary \ref{coro1} is terminated. \hfill $\square$

\subsection*{Proof of Theorem \ref{th2}} Let $A\subset\R^n$ be a nonempty and closed set, and let $r>0$. Assume that $A$ is $r$-strongly convex. Hence, it is evident that $A$ is both convex and bounded.  If $A$ is not $r$-negatively $\E_r(A)$-convex, then there exist two distinct points $a, a'\in\bdry A$, two unit vectors $\zeta_a \in N_A(a)$ and $\zeta_{a'}\in N_A(a')$, and two positive real numbers $t>r$ and $t'>r$, such that $$x=a-t\zeta_a=a'-t'\zeta_{a'}\in \E_r(A).$$
Since $A$ is $r$-strongly convex, we have that $$A\subset \widebar{B}(a-r\zeta_a;r).$$
This yields that $$\|a-r\zeta-y\|\leq r,\;\,\forall y\in A. $$
Add to this that $\|a-r\zeta-a\|=r$, we conclude that $$\df_A(a-r\zeta)=r\,\;\hbox{and}\,\;a\in\far_A(a-r\zeta_a).$$
We have \begin{eqnarray*} a\in\far_A(a-r\zeta_a)&\implies& a=\far_A(a-r\zeta_a+\tau r\zeta_a),\;\forall \tau>0 \\&\implies&   a=\far_A(a-(\tau+1)r\zeta_a),\;\forall \tau>0 \\&\implies& a=\far_A(a-t\zeta_a)\;\;\left(\hbox{for}\;\tau=\frac{t}{r}-1>0\right) \\ &\implies& a=\far_A(x).\end{eqnarray*}
Similarly, $a'=\far_A(x)$. This yields that $a=a'$, contradiction.

For the converse, assume that $A$ is convex, bounded, and $r$-negatively $\E_r(A)$-convex. In order to prove that $A$ is $r$-strongly convex, it is sufficient to prove that $\far_A(x)$ is a singleton for all $x\in\E_r(A)$, see Proposition \ref{prop0}. Let $x\in\E_r(A)$. If $\far_A(x)$ is not a singleton, then it contains two distinct points $a$ and $a'$. Let $\zeta_a:=\frac{a-x}{\|a-x\|}\in N_A(a)$ and $\zeta_{a'}:=\frac{a'-x}{\|a'-x\|}\in N_A(a')$. We have $$x=a-\|a-x\|\zeta_a= a'-\|a'-x\|\zeta_{a'}\;\,\hbox{and}\;\,\|a-x\|=\|a'-x\|=\df_A(x)>r,$$
which contradicts the $r$-negative $\E_r(A)$-convexity of $A$. The proof of Theorem \ref{th2} is terminated. \hfill $\square$


\begin{thebibliography}{99}


\bibitem{balashov} M. V. Balashov, G. E. Ivanov, On farthest points of sets, Math. Notes, 80 (2006), 159--166.

\bibitem{cannarsa} P. Cannarsa, C. Sinestrari:  {\em Semiconcave Functions, Hamilton-Jacobi Equations, and Optimal Control}, Birkh\"auser, Boston (2004).

\bibitem{canino}  A. Canino: {\em On $p$-convex sets and geodesics}, J. Diff. Equations 75/1 (1988) 118--157.

\bibitem{csw} F. H. Clarke, R. J. Stern, P. R. Wolenski:  {\em Proximal smoothness and the lower-$C^2$ property}, J. Convex Analysis 2/1-2 (1995) 117--144.

\bibitem{clsw} F. H. Clarke, Yu. Ledyaev, R. J. Stern, P. R. Wolenski:  {\em Nonsmooth Analysis and Control Theory}, Graduate Texts in Mathematics 178, Springer, New York (1998).

\bibitem{cm} G. Colombo, A. Marigonda:  {\em Differentiability properties for a class of non-convex functions}, Calc. Var. 25 (2005) 1–31.

\bibitem{fed} H. Federer:  {\em Curvature measures}, Trans. Amer. Math. Soc. 93 (1959) 418--491.

\bibitem{frankowska} H. Frankowska, C. Olech: {\em $R$-convexity of the integral of set-valued functions}, Contributions
to Analysis and Geometry, Johns Hopkins Univ. Press, Baltimore, Md., (1981) 117--129,

\bibitem{mord} B. S. Mordukhovich:  {\em Variational Analysis and Generalized Differentiation}. I: Basic Theory, Springer, Berlin (2006).

\bibitem{Goncharov} V. V. Goncharov, G. E. Ivanov: {\em Strong and weak convexity of closed sets in a Hilbert
space}, Operations Research, Engineering, and Cyber Security, Springer Optim. Appl., 113, Springer, Cham, (2017), 259--297.

\bibitem{Nacry2025} F. Nacry, L. Thibault: {\em Strongly convex sets with variable radii}, Mathematical Control and Related Fields, doi:10.3934/mcrf.2024044.

\bibitem{Nakajima} S. Nakajima, {\it \"Uber konvexe Kurven and Fl\"aschen}, Tohoku Mathematical Journal, voll. 29 (1928), 227--230.

\bibitem{JCA2009} C. Nour, R. J. Stern, J. Takche:  {\em Proximal smoothness and the exterior sphere condition}, J. Convex Analysis 16/2 (2009) 501--514.

\bibitem{JCA2018b} C. Nour, J. Takche:  {\em A new class of sets regularity}, J. Convex Analysis 25/4 (2018) 1059--1074.

\bibitem{penot}  J.-P. Penot:  {\em Calculus Without Derivatives}, Graduate Texts in Mathematics 266, Springer, New York (2013).

\bibitem{prt} R. A. Poliquin, R. T. Rockafellar, L. Thibault:  {\em Local differentiability of distance functions}, Trans. Amer. Math. Soc. 352 (2000) 5231--5249.

\bibitem{rock} R.T. Rockafellar: {\em Clarke’s tangent cones and the boundaries of closed sets in $\R^n$}, Nonlinear Anal., Theory Methods Appl. 3 (1979) 145--154.

\bibitem{rockwet} R. T. Rockafellar, R. J.-B. Wets:  {\em Variational Analysis}, Grundlehren der Mathematischen Wissenschaften 317, Springer, Berlin (1998).

\bibitem{schneider} R. Schneider: {\em Convex Bodies: The Brunn-Minkowski Theory}, 2nd ed. Cambridge University Press (2013).

\bibitem{shapiro} A. S. Shapiro:  {\em Existence and differentiability of metric projections in Hilbert spaces}, SIAM J. Optim. 4 (1994) 231--259.

\bibitem{thibault} L. Thibault:  {\em Unilateral Variational Analysis in Banach Spaces}, World Scientific (2023).

\bibitem{Tietze} H. Tietze: {\em \"Uber Konvexheit im kleinen und im gro\ss en und \"uber gewisse den Punkten einer Menge zugeordnete Dimensionszahlen}, Math. Z. 28 (1928) 697--707.

\bibitem{weber} A. Weber, G. Reissig: {\em Local characterization of strongly convex sets}, J. Math. Anal. Appl. 400, (2013)
743--750.

\end{thebibliography}
\end{document}